\documentclass[12pt]{amsart}
\usepackage{fullpage}
\usepackage[all]{xy}
\usepackage{hyperref, amscd, amssymb, amsmath, amsthm, graphics}
\usepackage{amsmath,amsfonts,amsthm,amssymb}
\usepackage{latexsym,amsmath}
\usepackage{graphicx,psfrag}
\usepackage{pinlabel}
\usepackage{mathrsfs}
\usepackage{xcolor}

\newtheorem{theorem}{Theorem}[section]
\newtheorem{proposition}[theorem]{Proposition}
\newtheorem{corollary}[theorem]{Corollary}
\newtheorem{lemma}[theorem]{Lemma}
\newtheorem{claim}[theorem]{Claim}
\newtheorem{assumption}[theorem]{Assumption}

\theoremstyle{definition}
\newtheorem{definition}[theorem]{Definition}

\theoremstyle{remark}
\newtheorem{remark}[theorem]{Remark}

\newtheorem{question}{Question}
\numberwithin{equation}{section}
\newcommand{\bZ}{\mathbb Z}
\newcommand{\bR}{\mathbb R}

\newcommand{\bH}{\mathbb H}
\newcommand{\mminus}{\backslash\backslash}
\newcommand{\Mod}{\operatorname{Mod}}
\newcommand{\Homeo}{\operatorname{Homeo}}
\newcommand{\from}{\colon\thinspace}

\newcommand{\mc}{\mathcal}
\newcommand{\Interior}[1]{\mathring{#1}}
  
\begin{document}
\sloppy

\title[Pseudo-Anosov surfaces and Dynamics in 3-manifolds]
{Pseudo-Anosov surfaces and Dynamics in 3-manifolds}

\author{Jason Manning}
\address{Department of Mathematics, Cornell University, Ithaca, NY 14853, USA}
\email{jfmanning@cornell.edu}
\author{Christoforos Neofytidis }
\address{Department of Mathematics and Statistics, University of Cyprus, Nicosia 1678, Cyprus}
\email{neofytidis.christoforos@ucy.ac.cy}

\subjclass[2010]{57M50, 57N37, 57R22, 51H20}
\keywords{pseudo-Anosov surface, partially pseudo-Anosov dynamics, mapping torus, I-bundle}

\date{\today}

\begin{abstract}
  We determine which closed orientable $3$--manifolds $M$ admit a self-homeomorphism restricting to a pseudo-Anosov map on an incompressible subsurface $\Sigma$, which we call a \emph{pseudo-Anosov surface}.  
  When $M$ is irreducible, we show that the self-homeomorphism of $M$ is isotopic rel $\Sigma$ to a ``partially pseudo-Anosov'' homeomorphism, a notion that we will introduce.  This is motivated by the corresponding results for Anosov tori in irreducible $3$--manifolds, and the connection to partially \emph{hyperbolic} diffeomorphisms, obtained by F. Rodriguez-Hertz, J. Rodriguez-Hertz and R. Ures.
\end{abstract}

\maketitle

\section{Introduction}

A self-diffeomorphism $f$ of a manifold $M$ is \emph{hyperbolic} or \emph{Anosov} if the tangent bundle $TM$ splits into two $f$--invariant sub-bundles (the \emph{stable} and \emph{unstable} sub-bundles) in such a way that the derivative of the diffeomorphism is uniformly contracting on the stable sub-bundle and uniformly expanding on the unstable sub-bundle.  The diffeomorphism $f$ is said to be \emph{partially hyperbolic} if $TM$ splits into three sub-bundles, which are a stable and unstable bundle as above, and a \emph{central} bundle on which the derivative behaves in a way which is intermediate between the behavior on the stable and unstable bundles.  

Rodriguez-Hertz--Rodriguez-Hertz--Ures~\cite{RHRHU} 
showed that if $M$ is an irreducible $3$--manifold, and $f$ is a self-diffeomorphism of $M$ which restricts to an Anosov map on some embedded torus $T$ (such $T$ is called an \emph{Anosov torus}), then $T$ must be incompressible and $M$ must be the mapping torus of $T$.  In this case, the monodromy of the mapping torus is itself either an Anosov map or $\pm id_T$.  Moreover, their argument implies that $f$ is homotopic rel $T$ to a partially hyperbolic diffeomorphism.
In particular, the hyperbolic dynamics on $T$ can be extended to partially hyperbolic dynamics on the entire $3$--manifold.  Conversely, the fiber $T$ of each of the above mapping tori $M$ is an {\em Anosov torus} in $M$.

In this paper we investigate a similar phenomenon where $T$ is replaced by a higher genus surface $\Sigma$, and Anosov maps are replaced by pseudo-Anosov maps. This leads us to introduce the concept of {\em partially pseudo-Anosov homeomorphism} (a precise definition will be given in Section \ref{sec:definitions}) in analogy to partially hyperbolic homeomorphism when dealing with tori. 
Our main result is the following:

\begin{theorem}\label{thm:classify}
  Let $M$ be a closed orientable irreducible $3$--manifold. There is a homeomorphism $f\from M\to M$ which restricts to a pseudo-Anosov map on an invariant embedded incompressible two-sided surface $\Sigma$ if and only if
   \begin{enumerate}
    \item\label{itm:pamt} $M$ is the mapping torus of a pseudo-Anosov map $\varphi\from \Sigma\to\Sigma$; or 
    \item\label{itm:sfs} $M$ is Seifert fibered with large base orbifold and $\Sigma$ is horizontal.
 \end{enumerate}
 Any such $f$ is isotopic rel $\Sigma$ to a partially pseudo-Anosov homeomorphism.
 \end{theorem}
For us an orbifold is \emph{large} if it admits a pseudo-Anosov automorphism.  See Subsection~\ref{ss:large} for more detail.
 \begin{remark}
   Suppose a self-homeomorphism $f$ of $M$ restricts to a pseudo-Anosov map on some \emph{one-sided} surface $\Sigma$.  Let $N$ be a regular neighborhood of $\Sigma$.  Then $f$ is isotopic (via an isotopy supported near $N$) to a map $f'$ which restricts to a pseudo-Anosov map on the two-sided surface $\partial N$.
 \end{remark}

When a surface $\Sigma$ as in Theorem \ref{thm:classify} exists, we will say that $M$ admits a {\em pseudo-Anosov surface} and that $f$ \emph{realizes} $\Sigma$. From a topological point of view, our motivation comes from the work by Wang and the second author~\cite{NeoWang}, where it was shown that if a reducible 3-manifold $M$ contains an incompressible surface which is invariant under a homeomorphism $f$ of $M$, then this surface lies in a prime summand preserved by a homeomorphism isotopic to $f$. Whether this was the case for Anosov tori had been asked  by Rodriguez-Hertz--Rodriguez-Hertz--Ures~\cite{RHRHU}. 
Combining~\cite{NeoWang} with Theorem \ref{thm:classify}, we obtain a complete list of 3-manifolds that admit 
a pseudo-Anosov surface:

\begin{corollary}\label{c:main}
A closed orientable reducible 3-manifold admits 
a pseudo-Anosov surface if and only if one of its prime summands belongs to (1) or (2) as given in Theorem \ref{thm:classify}.
\end{corollary}

The hypothesis of incompressibility for Anosov tori was not present in the work of Rodriguez-Hertz--Rodriguez-Hertz--Ures (because, as mentioned above, Anosov tori must be incompressible), but it is essential here.  If we remove the incompressibility requirement on the pseudo-Anosov surface, no conclusions about the $3$--manifold can be obtained.  Indeed there are homeomorphisms of $S^3$ restricting to a pseudo-Anosov map of a genus $2$ Heegaard surface.  Thus, for any $3$--manifold, there is a self-homeomorphism which is arbitrarily close to the identity, but which restricts to a pseudo-Anosov map of a compressible subsurface; see Section \ref{sec:compressible}.

\subsection*{Outline of the paper}
In Section \ref{sec:definitions} we give formal definitions and some background on the proposed notion of partially pseudo-Anosov dynamics. 
In Section \ref{part1} we prove the major part of Theorem \ref{thm:classify}, showing that the existence of a pseudo-Anosov surface $\Sigma$ in an irreducible 3-manifold implies that the 3-manifold must belong to \eqref{itm:pamt} or \eqref{itm:sfs} as in Theorem \ref{thm:classify}, and that any map that realizes $\Sigma$ as a pseudo-Anosov surface
is isotopic relative to $\Sigma$ to a partially pseudo-Anosov homeomorphism. In Section \ref{sec:existence} we complete the proof of Theorem \ref{thm:classify} by giving examples of pseudo-Anosov surfaces in manifolds as in \eqref{itm:pamt} and \eqref{itm:sfs} of Theorem \ref{thm:classify}.
In Section \ref{sec:reducible} we discuss briefly the case of reducible 3-manifolds (Corollary \ref{c:main}). Finally, in Section \ref{sec:compressible} we give examples of compressible pseudo-Anosov surfaces in all 3-manifolds.

\subsection*{Acknowledgments}
We thank Brian Bowditch, Kathryn Mann, Dan Margalit, Federico Rodriguez-Hertz, Saul Schleimer and Bena Tshishiku for useful conversations.  Thanks to Shicheng Wang for a helpful comment on the work of Bonahon.
C. Neofytidis gratefully acknowledges the hospitality of Cornell University, MPIM Bonn, the University of Geneva and the University of Warwick, where parts of this project were carried out.  J. Manning was partially supported by Simons Collaboration Grants \#942496 and \#524176.

\section{Background}\label{sec:definitions}

We first recall some definitions from smooth dynamics, and propose the notion of partially pseudo-Anosov homeomorphism.

\begin{definition}
  Suppose $M$ is a Riemannian manifold.  A diffeomorphism $\alpha\from M\to M$ is \emph{hyperbolic} if there is an $\alpha$--invariant splitting of the tangent bundle $TM = E^s\oplus E^u$ and some $n>0$ so that for any unit vectors $v^s\in E^s, v^u\in E^u$ we have
  \[ \|D \alpha^n(v^s)\| \le \frac{1}{2} \quad\mbox{and}\quad \|D \alpha^n(v^u)\| \ge 2.\]
  The bundles $E^s$ and $E^u$ are called the \emph{stable} and \emph{unstable} sub-bundles of $TM$, respectively.

  The diffeomorphism is \emph{partially hyperbolic} if there is a splitting of the tangent bundle into \emph{three} invariant sub-bundles $TM = E^s\oplus E^c \oplus E^u$, and some $n>0$ so that for any unit vectors $v^s\in E^s, v^c\in E^s, v^u\in E^u$ we have
  \[ \|D \alpha^n(v^s)\|\le \frac{1}{2}\min\{1,\|D\alpha^n(v^c)\|\}\quad\mbox{and}\quad 
\|D \alpha^n(v^u)\|\ge 2\max\{1,\|D\alpha^n(v^c)\|\}.\]
\end{definition}
Hyperbolicity has a number of equivalent definitions in the literature.  Ours is taken from~\cite{BFP23}.  See \cite{CHHU} for more details.

\begin{definition}A homeomorphism $\psi$ of a surface is called \emph{pseudo-Anosov} if it preserves a pair of (singular) transverse measured foliations $(\mathcal{F}^s,\mu^s)$ and $(\mathcal{F}^u,\mu^u)$, and there is a number $\lambda>1$ so that for any arc $\alpha$ transverse to $\mathcal{F}^s$, we have $\mu^s(\psi(\alpha)) = \lambda^{-1}\mu^s(\alpha)$ and for any arc $\beta$ transverse to $\mathcal{F}^u$ we have $\mu^u(\psi(\beta)) = \lambda \mu^u(\beta)$.
\end{definition}

  A pseudo-Anosov map is not $C^1$ at the singularities, despite the common terminology ``pseudo-Anosov diffeomorphism'' (see \cite[p. 175]{FLP}).  However, away from the singularities it is a (hyperbolic) diffeomorphism.  

\begin{definition}\label{def:pA}
 An embedded closed orientable two-sided incompressible surface $\Sigma$ in a $3$--manifold $M$ is called {\em pseudo-Anosov} if there exists a homeomorphism $f$ of $M$ such that $f(\Sigma)=\Sigma$ and $f\vert_\Sigma$ is a pseudo-Anosov map.  
 In that case, we say that $\Sigma$ is \emph{realized by $f$} and that $M$ admits a pseudo-Anosov surface.
\end{definition}

 Theorem \ref{thm:classify} classifies the pairs $(M,\Sigma)$ where $M$ is an irreducible $3$--manifold and $\Sigma$ is 
 a pseudo-Anosov surface in $M$.  In case it is, we show that there is a realizing homeomorphism $f$ which extends that dynamics throughout $M$ and has the same relationship to ``pseudo-Anosov dynamics'' as a partially hyperbolic automorphism has to hyperbolic dynamics.  We propose the following definition to capture what we mean by this.

\begin{definition}\label{def:ppa}
  A homeomorphism $f\from M\to M$ is \emph{partially pseudo-Anosov} if there is an invariant embedded $1$-manifold $C \subset M$ so that $f$ is partially hyperbolic on the complement of $C$.
\end{definition}

The following result of McCarthy will be important
 for both directions of our main theorem. 
 We thank Bena Tshishiku for pointing out that the statement we give here follows from McCarthy's proof.

 \begin{theorem}\label{thm:mccarthy}\cite{McC}
  Let $\psi\from \Sigma\to \Sigma$ be pseudo-Anosov, and let $\mc{F}^s,\mc{F}^u$ be the stable and unstable foliations of $\psi$.
  Let $N$ be the normalizer of $[\psi]$ in $\Mod^\pm(\Sigma)$, and let $\mc{G}$ be the stabilizer of the pair $\{\mc{F}^s,\mc{F}^u\}$ in $\Homeo(\Sigma)$.  The projection $\Homeo(\Sigma)\to \Mod^\pm(\Sigma)$ restricts to an injection $\mc{G}\to \Mod^{\pm}(\Sigma)$ whose image contains $N$.
  Moreover
  \[ [N:\langle [\psi]\rangle] \le [\mc{G}:\langle \psi\rangle] <\infty .\]
\end{theorem}
\begin{proof}
  The fact that $N$ is virtually generated by $[\psi]$ is part of \cite[Theorem 1]{McC}.

  The rest of the statement follows from McCarthy's proof.  Indeed,
  \cite[Lemma 1]{McC} shows that each $s\in N$ is realized by a homeomorphism $\tilde{s}$ either preserving or exchanging the stable and unstable foliations $\mathcal{F}^s$ and $\mathcal{F}^u$ of $\psi$. 
  
  In particular, any element $s\in N$ is realized by $\tilde{s}\in \Homeo(\Sigma)$ which lies in the subgroup $\mathcal{G}$ preserving the pair of foliations $\{\mathcal{F}^u,\mathcal{F}^s\}$.  McCarthy's \cite[Lemma 4]{McC} says
  that the map from $\Homeo(\Sigma)$ to $\Mod(\Sigma)$ restricts to an injection on $\mathcal{G}$.  (That lemma is stated for orientation preserving maps, but the proof does not use the orientation.)
  It follows that the element $\tilde{s}$ is uniquely determined by $s$, and so $\mathcal{G}$ contains the desired realization of $N$.  McCarthy's Lemmas~\cite[Lemma 2, Lemma 3]{McC} show that $\mc{G}$ is virtually cyclic.
\end{proof}

\subsection{Large orbifolds}\label{ss:large}

 As noted in the introduction, we call a $2$--orbifold $\mathcal{O}$ \emph{large} if there exists a pseudo-Anosov map from $\mathcal{O}$ to itself.
  We only discuss $2$--orbifolds with isolated orbifold locus (i.e. no \emph{mirrors}) and compact underlying surface.  If $\mathcal O$ is such a $2$--orbifold, let $F_{\mathcal O}$ be the (punctured) surface obtained by deleting the orbifold locus. 
  The orbifold $\mathcal O$ admits a pseudo-Anosov self-map if and only if the surface $F_{\mathcal O}$ admits such a map.  It is well-known that a finite type orientable surface $F$ admits a pseudo-Anosov map if and only if it is either a punctured torus, or has Euler characteristic at most $-2$.  Slightly less well-known is the fact that a finite type non-orientable surface $F$ admits a pseudo-Anosov map if and only if its Euler characteristic is at most $-2$ (see \cite[Table 3]{Strenner}).

\begin{lemma}\label{lem:descends}
  Let $\psi\from \Sigma\to\Sigma$ be a pseudo-Anosov map, and let $G$ be a finite subgroup of the centralizer of $\psi$ in $\mathcal{G}$, the subgroup of $\Homeo(\Sigma)$ described in Theorem~\ref{thm:mccarthy}.  Then the quotient orbifold $\mathcal{O} = \Sigma/G$ is large.
\end{lemma}
\begin{proof}
    First observe that $\psi$ must send $G$--orbits to $G$--orbits so it induces some automorphism $\eta$ of $\mathcal{O}$.  
  Next note that the centralizer of $\psi$ in $\mathcal{G}$ must preserve the \emph{ordered} pair $(\mathcal{F}^u,\mathcal{F}^s)$.  In particular, the stable and unstable foliations of $\psi$ descend to singular foliations on
  $\mathcal{O}$.  Moreover if $\rho$ is the dilatation of $\psi$, then $\eta$ multiplies the transverse measures on these by $\rho^{\pm 1}$.  In particular $\eta$ is pseudo-Anosov.
\end{proof}

\section{Determining 3-manifolds that can admit pseudo-Anosov surfaces}\label{part1}

The main goal of this section is to prove the following:

\begin{theorem}\label{thm:necessary}
 Let $M$ be a closed orientable irreducible $3$--manifold. If $M$ admits a pseudo-Anosov surface $\Sigma$, realized by $f\colon M\to M$, then 
   \begin{enumerate}
    \item\label{itm:pamtA} $M$ is the mapping torus of a pseudo-Anosov map $\varphi\from \Sigma\to\Sigma$ which commutes with $(f|_{\Sigma})^2$; or
    \item\label{itm:sfsB} $M$ is Seifert fibered with large base orbifold and $\Sigma$ is horizontal.
 \end{enumerate}
 Moreover, $f$ is isotopic rel $\Sigma$ to a partially pseudo-Anosov homeomorphism.
 \end{theorem}

 The proof of Theorem \ref{thm:classify} will then be completed in Section \ref{sec:existence} by constructing examples of pseudo-Anosov surfaces in $3$--manifolds as given by Theorem \ref{thm:necessary} (\ref{itm:pamtA}) and (\ref{itm:sfsB}).
 
 In the rest of this section, we make the following

\begin{assumption}\label{pA}
  $M$ is a closed orientable irreducible 3-manifold which admits a pseudo-Anosov surface $\Sigma$ realized by a homeomorphism $f$.  The restriction of $f$ to $\Sigma$ will be denoted by $\psi$.
\end{assumption}

\subsection{A necessary condition}\label{sec:jsj}

First, we will show that $M\mminus \Sigma$ is an $I$--bundle, where $M\mminus \Sigma$ is the path completion of $M\setminus \Sigma$.\footnote{The manifold $M\mminus \Sigma$ is sometimes defined to be $M$ minus the interior of a regular neighborhood of $\Sigma$.  It is important later that a homeomorphism on $M$ preserving $\Sigma$ induces a specific homeomorphism of $M\mminus \Sigma$, not just a homeomorphism up to homotopy.}
Our tool is the characteristic submanifold from JSJ theory.

\subsubsection{Characteristic Seifert pairs and hyperbolicity}
We begin by recalling briefly some fundamental results and definitions in 3-manifold topology for characteristic Seifert pairs, 
following Jaco~\cite{Jaco}, Jaco-Shalen~\cite{JS} and Johannson~\cite{Jo}.  For more details, we refer the reader to the above references.

A compact 3-manifold $M$ (possibly with boundary) is called \emph{Haken} if it is orientable, irreducible, and contains a two-sided incompressible surface.  For example, any orientable irreducible 3-manifold with non-empty boundary is Haken.
A  \emph{Haken-manifold pair} $(M,F)$ is a Haken 3-manifold $M$ with an incompressible surface $F$ in the boundary $\partial M$.  When $M$ is connected, $(M,F)$ is called a \emph{Seifert pair} if either
\begin{enumerate}
\item $M$ is an $I$--bundle over a surface and $F$ is the corresponding $\partial I$--bundle; or
\item $M$ is Seifert fibered and $F$ is a union of fibers.
\end{enumerate}
In general, a Seifert pair is a union of such components.  

A fundamental result of Jaco--Shalen~\cite{JS} and Johannson~\cite{Jo} is that any Haken $3$--manifold contains a unique (up to isotopy) maximal Seifert pair:
\begin{theorem}[Characteristic submanifold theorem]\label{thm:JSJ}\cite[IX.12 and IX.15]{Jaco}
  Let $M$ be a Haken $3$--manifold with incompressible (possibly empty) boundary.  Then, up to isotopy, there is a unique maximal Seifert pair $(\chi(M),F)\subset(M,\partial M)$ with essential boundary.  Moreover, any essential annulus or torus in $M$ can be isotoped into $\chi(M)$.
\end{theorem}
This maximal Seifert pair is also called the \emph{characteristic submanifold} of $M$.

A manifold $M$ is  \emph{atoroidal} if it does not contain any essential tori and \emph{acylindrical} if it does not contain any essential annuli.  An atoroidal and acylindrical Haken manifold is necessarily \emph{simple} in the sense that its characteristic submanifold is empty.

\subsubsection{Reduction to $I$-bundles}

Using Jaco--Shalen and Johannson's Theorem~\ref{thm:JSJ}, we will show that the complement of a pseudo-Anosov surface in an irreducible $3$--manifold is an $I$--bundle.  Bonahon uses essentially the same argument to prove something slightly different in~\cite[Section 5]{Bon}.

 \begin{proposition}\label{prop:I-bundle}
 Suppose Assumption \ref{pA} holds. Then $\hat{M}=M\mminus\Sigma$ is an $I$-bundle.
 \end{proposition}
\begin{proof}

  First, observe that $\hat{M}$ is orientable, irreducible with incompressible boundary, and $f$ induces a homeomorphism
  $\hat f\colon \hat{M}\to \hat{M}$.  The map $\hat{f}$ may exchange the boundary components of $\hat{M}$, but the map $\hat{f}^2$ preserves them, and restricts to a pseudo-Anosov map on each one.

By Theorem \ref{thm:JSJ}, there is a unique characteristic Seifert pair 
$$(\chi(\hat{M}),F)\subset(\hat{M},\partial\hat{M}).
$$
For brevity, we write $\chi$ for $\chi(\hat{M})$.  
We may assume by changing $\hat{f}$ by an isotopy that $\hat{f}(\chi) = \chi$.   Suppose that $F\cap P$ is non-empty for some boundary component $P$ of $\hat{M}$.  Then $\partial F\cap P$ is a system of curves on $P$ preserved by the pseudo-Anosov map $\hat{f}^2|_N$.  Such a system can only be empty, so $P\subset F$.  In other words, the frontier of $\chi$ is disjoint from $\partial \hat{M}$, and thus consists only of tori.

\begin{claim}\label{claim}
  If $P$ is a component of $\partial\hat{M}$, then $P\subset\chi=\chi(\hat{M})$.
\end{claim}
\begin{proof}[Proof of Claim] 
  Suppose the contrary, and let $M_0$ be the component of $\hat{M}\setminus\chi$ containing $P$.  Since $\hat{f}(\chi) = \chi$ and $\hat{f}^2(P) = P$, we have $\hat{f}^2(M_0) = M_0$.  Since $\hat{f}|_P$ is infinite order in the mapping class group of $P$, and $P$ is incompressible in $M_0$, this implies $\hat{f}^2$ is infinite order in the mapping class group of $M_0$.  However, $M_0$ is acylindrical and atoroidal 
  by~\cite[Lemma X.17 and Remark X.16(b)]{Jaco}, and so it is simple.  Johannson showed that the mapping class group of a simple Haken $3$--manifold is finite \cite[\S 27]{Jo}.  This contradiction establishes the claim.
\end{proof}
By Claim \ref{claim}, each component $P$ of the boundary of $\hat{M}$ is contained in some component of $\chi$.  Since the genus of $\Sigma$ is at least two, the component of $\chi$ must be an $I$--bundle, and in fact must be the entire component of $\hat{M}$ containing $P$.
\end{proof}

\subsection{Proof of Theorem \ref{thm:necessary}}\label{sec:incompressible}

In order to prove Theorem \ref{thm:necessary}, 
we need to show that the pair $(M,\Sigma)$ is as described in Theorem~\ref{thm:necessary} \eqref{itm:pamtA}-\eqref{itm:sfsB}, and that $f$ is isotopic rel $\Sigma$ to a partially pseudo-Anosov homeomorphism.

Note that, if $f$ exchanges the two sides of $\Sigma$, then $f^2$ preserves the two sides and still restricts to a pseudo-Anosov map on $\Sigma$.  

Our argument splits into two cases, depending on whether $\Sigma$ is separating or not.

\subsubsection{The non-separating case}\label{s:non-separating}

In case $\Sigma$ is non-separating, then Proposition~\ref{prop:I-bundle} implies that $M$ is a mapping torus 
\begin{equation}\label{eq:mt}
M_\varphi = \Sigma \times [0,1] / (x,0)\sim (\varphi(x),1)
\end{equation}
for some homeomorphism $\varphi\colon\Sigma\to\Sigma$, where we identify $\Sigma\subset M$ with $\Sigma\times\{0\}\subset M_\varphi$.  Since the mapping torus only depends on the isotopy class of the monodromy $\varphi$, the Nielsen--Thurston classification~\cite{Th2,FLP,Bers78,BestHand} tells us that we may take $\varphi$ to be one of following three mutually exclusive types:  Either
\begin{itemize}
\item[(a)] $\varphi$ is periodic (possibly the identity);
\item[(b)] $\varphi$ is reducible and infinite order; or
\item[(c)] $\varphi$ is pseudo-Anosov.
\end{itemize}
Notice that any nonzero power of $\varphi$ will be of the same type.

We first observe the following:

\begin{lemma}\label{lem:commute}
  The mapping class $[\varphi]$ commutes with $[\psi]^2$.
\end{lemma}
\begin{proof}
  We replace $M$ by the mapping torus $M_\varphi$ as described in equation~\eqref{eq:mt}, with $\Sigma = \Sigma\times\{0\}$, and lift $f$ to $\hat f\from \Sigma\times[0,1]\to \Sigma\times[0,1]$.  Since $f^2$ preserves the sides of $\Sigma$, we have $\hat{f}^2(x,0) = (\psi^2(x),0)$ for any $x\in \Sigma$, and,
  since $\hat{f}^2$ must respect the equivalence relation $\sim$, we have
  \[ \hat{f}^2(x,1)\sim \hat{f}^2(\varphi^{-1}(x),0) = (\psi^2\varphi^{-1}(x),0)\sim (\varphi\psi^2\varphi^{-1}(x),1) .\]
  In particular  $\hat{f}^2 (x,1) = (\varphi\psi^2 \varphi^{-1}(x),1)$ for each $x\in \Sigma$.
  Let $\pi\from \Sigma\times [0,1]$ be the projection onto $\Sigma$, and $H = \pi\circ \hat{f}^2$.  Then $H$ is a homotopy with $H_0 = \psi^2$ and $H_1 = \varphi \psi^2 \varphi^{-1}$.  Homotopic surface homeomorphisms are isotopic, so we are finished.
\end{proof}

McCarthy's Theorem~\ref{thm:mccarthy} implies that any infinite order element of the normalizer of $[\psi]^2$ has a positive power which is also a power of $[\psi]^2$.  In particular it would also be pseudo-Anosov, so we have the following.
\begin{corollary}\label{cor:notior}
  The monodromy $\varphi$ is not infinite order reducible.
\end{corollary}

\begin{lemma}\label{lem:betterphi}
  After possibly replacing $\varphi$ by a homotopic map, one of the following holds:
  \begin{enumerate}
  \item\label{itm:preserve} $f$ preserves the sides of $\Sigma$ and $\varphi \psi = \psi\varphi$; or
  \item\label{itm:exchange} $f$ exchanges the sides of $\Sigma$ and $\varphi \psi = \psi\varphi^{-1}$.
  \end{enumerate}
  In either case $\psi^2$ commutes with $\varphi$.
\end{lemma}
\begin{proof}
  We apply McCarthy's theorem~\ref{thm:mccarthy} to $\psi^2$.  As in that statement, let $\mathcal{G}$ be the subgroup of $\Homeo(\Sigma)$ preserving the pair $\{\mathcal{F}^s,\mathcal{F}^u\}$, the stable and unstable foliations of $\psi^2$ (which are the same as those of $\psi$).  Let $N$ be the normalizer of $[\psi^2]$ in $\Mod^{\pm}(\Sigma)$.  
  Lemma~\ref{lem:commute} implies that $[\varphi]\in N$.  Theorem~\ref{thm:mccarthy} then tells us that $\varphi$ is homotopic to an element of $\mathcal{G}$.  For the remainder of the proof we assume that $\varphi\in \mathcal{G}$.
  
    We divide into cases.

    If $f$ preserves the sides of $\Sigma$, then the same argument as in Lemma~\ref{lem:commute} shows that $[\varphi]$ commutes with $[\psi]$ (not just $[\psi^2]$).  Because we have chosen $\varphi\in \mathcal{G}$, Theorem~\ref{thm:mccarthy} gives us $\varphi\psi = \psi\varphi$. 

    If $f$ exchanges the sides of $\Sigma$ we argue similarly.  We have an induced homeomorphism on $M\mminus \Sigma = \Sigma\times[0,1]$, but now it exchanges the boundary components.  
    Recall that $M_\varphi = \Sigma\times[0,1]/\sim$ where $(x,0) \sim (\varphi(x),1)$ for all $x$.  We deduce
  \[ \hat f (x,0) \sim (\psi(x),0) \sim (\varphi\psi(x),1), \]
  and
  \[ \hat f(x,1) \sim \hat f(\varphi^{-1}(x),0) = (\varphi\psi\varphi^{-1}(x),1)\sim (\psi\varphi^{-1}(x),0) \]
  If $\tau(x,t) = (x,1-t)$ flips $\Sigma\times[0,1]$ upside-down, we have
  \[ \tau \hat f (x,0) = (\varphi\psi(x),0) \quad\mbox{and}\quad \tau\hat f(x,1) = (\psi\varphi^{-1}(x),1).\]
  In other words $\tau\hat f$ gives a homotopy from $\varphi\psi$ to $\psi\varphi^{-1}$.  In other words $[\varphi\psi] = [\psi\varphi^{-1}]$.  Again our choice of $\varphi$ together with McCarthy's theorem imply this relation is satisfied in $\Homeo(\Sigma)$, so we have $\varphi\psi = \psi\varphi^{-1}$ in this case.
\end{proof}
For the rest of this sub-subsection we assume that $\varphi$ satisfies the conclusions of Lemma~\ref{lem:betterphi}.  By Corollary~\ref{cor:notior}, $\varphi$ is either finite order or pseudo-Anosov.

We will now prove Theorem \ref{thm:necessary} in the non-separating case:

\begin{proposition}\label{prop:nonseparating}
 Suppose Assumption \ref{pA} holds and $\Sigma$ is non-separating. Then $M=M_\varphi$, where  
   \begin{enumerate}
    \item\label{itm:pamt1} 
    $\varphi$ is pseudo-Anosov  
    and commutes with $\psi^2$; or
    \item\label{itm:sfs2} $\varphi$ is periodic, commutes with $\psi^2$, $\Sigma$ is horizontal in the Seifert fibration  of $M_\varphi$ and $\mathcal{O}=\Sigma/\langle \varphi \rangle$ is large. 
 \end{enumerate}
 Moreover, $f$ is isotopic rel $\Sigma$ to a partially pseudo-Anosov homeomorphism.
\end{proposition}

For establishing the ``moreover" part of Proposition \ref{prop:nonseparating}, we will also need the following lemma, which is a consequence of work of Waldhausen:

\begin{lemma}\label{lem:wald1}
  Let $g\from \Sigma\times I\to \Sigma\times I$ be equal to the identity on $\partial (\Sigma \times I)$.  Then $g$ is isotopic to the identity rel $\partial (\Sigma \times I)$.  
\end{lemma}
\begin{proof}
  First we claim that $g$ is homotopic to a fiber-preserving map $g'$.  Indeed, write $g(x,t) = (g_1(x,t),g_2(x,t))$, and let $H_s(x,t) = (g_1(x,t),(1-s)g_2(x,t) + st)$.  We have $H_0 = g$, and $H_1(x,t) = (g_1(x,t),t)$. 

  Now for each $t$, $g'_t = ( x\mapsto g_1(x,t) )$ is a homotopy equivalence of $\Sigma$ to itself, homotopic to the identity.  Fixing a hyperbolic structure on $\Sigma$ and lifting to $\bH^2$, we may use straight-line homotopies to continuously deform each $g'_t$ to the identity on $\Sigma$.  This process depends continuously on $g'_t$, so we obtain a homotopy of $g'$ to the identity on $\Sigma\times I$.

  A result of Waldhausen~\cite[Theorem 7.1]{Waldhausen} implies that since $g$ is homotopic to the identity rel boundary, it is isotopic to the identity rel boundary.
\end{proof}

\begin{proof}[Proof of Proposition \ref{prop:nonseparating}]
  As pointed out above, 
  Proposition~\ref{prop:I-bundle} implies that $M$ is a mapping torus $M_\varphi$ for some $\varphi\from \Sigma\to \Sigma$.
  We may assume $\varphi$ satisfies the conclusions of Lemma~\ref{lem:betterphi}, in particular $\varphi$ commutes with $\psi^2$ in $\Homeo(\Sigma)$.  By Corollary~\ref{cor:notior}, $\varphi$ is either pseudo-Anosov or periodic.
  If $\varphi$ is pseudo-Anosov, we obtain conclusion~\eqref{itm:pamt1}.

  If on the other had $\varphi$ is finite order,
  the interval fibers of $\Sigma\times[0,1]$ glue together in $M_\varphi$ to form the circles of a Seifert fibration in which the surface $\Sigma$ is horizontal.  The base orbifold of this Seifert fibration is $\mathcal{O} = \Sigma / \langle \varphi\rangle$.  By Lemma~\ref{lem:descends} the orbifold $\mathcal{O}$ is large and we have established conclusion~\eqref{itm:sfs2}.

  We are left to establish that $f$ is isotopic rel $\Sigma$ to a partially pseudo-Anosov homeomorphism.
  If $f$ preserves the sides of $\Sigma$, Lemma~\ref{lem:betterphi} tells us that $\varphi\psi = \psi\varphi$.
  On $M\mminus \Sigma = \Sigma\times [0,1]$, we have the induced map $\hat f(x,t)$
  where $\hat f(x,0) = (\psi(x),0)$ and $\hat f(x,1) = (\varphi\psi\varphi^{-1}(x),1) = (\psi(x),1)$.  In particular, if $\hat h(x,t) = (\psi(x),t)$ then $g = \hat h^{-1}\circ \hat f$ satisfies the hypotheses of Lemma~\ref{lem:wald1}.  That lemma gives an isotopy rel boundary from $g$ to the identity on $\Sigma \times [0,1]$.  Composing with $\hat h$ we obtain an isotopy rel boundary from $\hat f$ to $\hat h$.  Gluing back up, we have an isotopy rel $\Sigma$ from $f$ to a homeomorphism $h$ of $M_\varphi$ which preserves each fiber and acts as a pseudo-Anosov on that fiber.  This map $h$ is partially pseudo-Anosov.

  If the map $f$ switches the sides of $\Sigma$, Lemma~\ref{lem:betterphi} tells us that $\varphi\psi = \psi\varphi^{-1}$.
  We still get an induced map $\hat f$ on $M\mminus \Sigma = \Sigma\times[0,1]$, but now it exchanges the boundary components.
  Arguing as in the proof there, we have
  $\hat{f}(x,0) = (\varphi\psi(x),1)$ and $\hat{f}(x,1) = (\psi\varphi^{-1}(x),0) = (\varphi\psi(x),0)$.  Let $\sigma(x,t) = (\varphi\psi(x),1-t)$.  Then the Waldhausen Lemma~\ref{lem:wald1} implies that $\sigma^{-1}\circ \hat{f}$ is homotopic rel boundary to the identity.  Composing with $\sigma$ we obtain a homotopy rel boundary from $\hat{f}$ to $\sigma$.  Gluing back up, we have an isotopy rel $\Sigma$ from $f$ to a map $h\from M\to M$ which sends fibers to fibers.  The square of $h$ sends each fiber to itself and acts as a pseudo-Anosov on each fiber, so $h$ is a partially pseudo-Anosov homeomorphism.  (Note that $(\varphi\psi)^2 = \psi^2$.)
\end{proof}

\subsubsection{The separating case}
Now we suppose that $\Sigma$ separates $M$ into two $I$--bundles.  This is to say that $M$ is a \emph{semi-bundle} as studied for example by Cooper--Walsh \cite{CW}.  This means that $M$ is obtained by gluing together two nontrivial $I$--bundles along their boundary.  Such $I$--bundles can be described as follows.
\begin{lemma}\label{lem:Ibundle}
  Suppose that $C$ is a nontrivial $I$--bundle with boundary $\Sigma$.  Then there is a fixed-point free involution $\tau\from \Sigma\to \Sigma$ so that
  \[ C \cong C_\tau := \Sigma\times[0,1]/\sim, \quad  x\in \Sigma, (x,1)\sim (\tau(x),1).\]
  Moreover if $\tau$ and $\tau'$ are homotopic involutions, then there is a homeomorphism $C_\tau\to C_{\tau'}$ which restricts to the identity on $\Sigma\times\{0\}$.
\end{lemma}
\begin{proof}
  (Sketch) The involution is given by following the $I$--fibers from one point on the boundary to another.  It is clearly fixed-point free.

  For the second part, note first that the identity on $\Sigma\times\{0\}$ extends to a homotopy equivalence $C_\tau\to C_{\tau'}$.  Then apply~\cite[Theorem 6.1]{Waldhausen} to promote the homotopy equivalence to a homeomorphism.
\end{proof}

The freedom to replace $\tau$ by a homotopic involution together with McCarthy's Theorem gives us the following.
\begin{lemma}\label{lem:semibundle}
  There are involutions $\tau_\epsilon$ of $\Sigma$ with $\epsilon \in \{\pm 1\}$
  such that $M=C_{-1}\cup_\Sigma C_{1}$, and
  each $C_\epsilon$ is homeomorphic to the mapping cylinder of the covering map $\Sigma \to \Sigma / \langle \tau_\epsilon\rangle$.  Moreover the involutions $\tau_\epsilon$ can be chosen to satisfy the following, for each $\epsilon\in \{\pm 1\}$.
  \begin{enumerate}
  \item\label{itm:taucommute} If $f$ preserves the sides of $\Sigma$, then $\psi = \tau_\epsilon\psi\tau_\epsilon$
  \item\label{itm:tauswitch} If $f$ exchanges the sides of $\Sigma$, then $\psi^2 =\tau_\epsilon \psi^2\tau_\epsilon$
    and $\psi = \tau_{-\epsilon}\psi\tau_\epsilon$.
  \end{enumerate}
\end{lemma}
\begin{proof}
  Let $C_{\pm 1}$ be the two submanifolds of $M$ bounded by $\Sigma$, so $M\mminus \Sigma = C_{-1} \sqcup C_{1}$,
  and let $\tau_{\pm 1}$ be the involutions given by Lemma~\ref{lem:Ibundle}.  (We may replace these by homotopic involutions by the second part of Lemma~\ref{lem:Ibundle}.)  
  For each $\epsilon\in \{\pm 1\}$, let $\hat C_\epsilon$ be the double cover of $C_\epsilon$ which is a trivial $I$--bundle.  We can parametrize each $\hat C_\epsilon$ as $\Sigma\times[0,1]$, where $\Sigma\times \{0\}$ is a chosen lift of $\Sigma$, and $r_\epsilon(x,t) = (\tau_\epsilon(x),1-t)$ is the deck transformation.
  
  Suppose that $f$ preserves the sides of $\Sigma$.  Since $f|C_\epsilon$ preserves $\Sigma$ it lifts to some $\tilde{f}_\epsilon\from \hat C_\epsilon\to \hat C_\epsilon$.  This lift commutes with the deck transformation $r_\epsilon$, so we have, in coordinates,
  \[ \tilde{f}_\epsilon(x,0) = (\psi(x),0)\mbox{ and } \tilde{f}_\epsilon(x,1) = (\tau_\epsilon\psi\tau_\epsilon(x),1). \]
  Projecting $\hat C_\epsilon$ onto the $\Sigma$ factor thus gives a homotopy from $\psi$ to $\tau_\epsilon\psi\tau_\epsilon$.  Using Theorem~\ref{thm:mccarthy} and the second part of Lemma~\ref{lem:Ibundle}, we can replace the involutions $\tau_\epsilon$ with involutions which commute with $\psi$ in $\Homeo(\Sigma)$, thus establishing conclusion~\eqref{itm:taucommute}.

  If $f$ exchanges the sides of $\Sigma$, the above argument applied to $f^2$ allows us to choose the involutions $\tau_\epsilon$ to commute with $\psi^2$.  More precisely, McCarthy's Theorem~\ref{thm:mccarthy} together with the second part of Lemma~\ref{lem:Ibundle} allows us to assume each $\tau_\epsilon$ lies
  in the group $\mc{G}<\Homeo(\Sigma)$ preserving the stable and unstable foliations for $\psi^2$.  Note that $\psi\in \mc{G}$ as well.
  Since $f(\Sigma) = \Sigma$, the induced automorphism of $M\mminus\Sigma =  C_{-1}\sqcup C_1$ lifts to an automorphism $\tilde f$ of the trivial $I$--bundle cover $\widehat{M\mminus\Sigma} = \hat C_{-1}\sqcup \hat C_1$.
  Explicitly we write the cover in coordinates as
  \begin{equation}
    \widehat{M \mminus \Sigma } = \Sigma \times [0,1]\times\{\pm 1\},
  \end{equation}
  so that for each $\epsilon$
  \begin{equation}\label{tildef0} \tilde f(x,0,\epsilon) = (\psi(x),0,-\epsilon),\end{equation}
  and the map $r(x,t,\epsilon) = (\tau_\epsilon( x),1-t,\epsilon)$ gives the deck transformation of the cover $\widehat{M\mminus \Sigma} \to M\mminus \Sigma$.
    Since $\tilde f$ commutes with $r$, we have
    \begin{equation}\label{tildef1}
      \tilde f(x,1,\epsilon) = r \tilde f r(x,1,\epsilon) = r\tilde f (\tau_\epsilon( x),0,\epsilon)
       = r (\psi\tau_\epsilon(x), 0,-\epsilon) = (\tau_{-\epsilon}\psi\tau_\epsilon(x), 1,-\epsilon).
    \end{equation}
    Projecting onto the $\Sigma$ factor, we see that $\tilde{f}$ restricted to $\hat C_\epsilon$ gives a homotopy
    from $\psi$ to $\tau_{-\epsilon}\psi\tau_\epsilon$.

    Because $\{\psi,\tau_{\pm 1}\}\subset \mc{G}$, and McCarthy's Theorem~\ref{thm:mccarthy} tells us that $\mc{G}$ injects into $\Mod(\Sigma)$, the homotopy must be an equality, and we obtain
    \[ \psi = \tau_{-\epsilon}\psi\tau_\epsilon \]
    for each $\epsilon$, establishing conclusion~\eqref{itm:tauswitch}.
\end{proof}

\begin{lemma}\label{lem:finiteorder}
  $M$ is double covered by a mapping torus $M_\varphi$, where $\varphi\from \Sigma\to\Sigma$ has finite order.
\end{lemma}
\begin{proof}
  Let $\tau_{\pm 1}$ be the involutions from the conclusion of Lemma~\ref{lem:semibundle}.
  The semi-bundle $M$ is double covered by a mapping torus $M_\varphi$, where $\varphi = \tau_{-1}\tau_1$.  (This double cover can be constructed by gluing together the double covers $\hat C_{\pm 1}$ from the proof of Lemma~\ref{lem:semibundle}.)
  For each $\epsilon\in\{\pm 1\}$ and any $k$ we have
  \begin{equation}\label{eq:anticommute}\tau_\epsilon\varphi^k\tau_\epsilon = \varphi^{-k}.\end{equation}
  By way of contradiction, suppose that $\varphi$ has infinite order.  By McCarthy's Theorem~\ref{thm:mccarthy} the normalizer of $\psi^2$ is virtually generated by $\psi^2$.  In particular there are nonzero $l,m$ so that 
  $\varphi^l = (\psi^2)^m$.  But then by Lemma~\ref{lem:semibundle}, the $\tau_{\pm 1}$ commute with $\varphi^l$, contradicting~\eqref{eq:anticommute}.
\end{proof}

We will now prove Theorem \ref{thm:necessary} in the separating case:

\begin{proposition}\label{prop:separating}
  Suppose Assumption \ref{pA} holds and $\Sigma$ is separating. Then
    $M$ is Seifert fibered with large base orbifold and $\Sigma$ horizontal.
 Moreover, $f$ is isotopic rel $\Sigma$ to a partially pseudo-Anosov homeomorphism.
\end{proposition}

For establishing the ``moreover" part of Proposition \ref{prop:separating}, we will also need the following refinement of Lemma \ref{lem:wald1}:

\begin{lemma}\label{lem:wald2}
  Let $\tau\from \Sigma\to \Sigma$ be a fixed point free involution, and let $\rho\from \Sigma\times I\to\Sigma\times I$ be given by $\rho(x,t) = (\tau(x),1-t)$.  
  If $g$ as in Lemma~\ref{lem:wald1} commutes with $\rho$, then the isotopy can be chosen to descend to the quotient manifold $Q = \Sigma\times I/\langle \rho\rangle$.
\end{lemma}
\begin{proof}
  Referring to the proof of Lemma~\ref{lem:wald1}, it is easy to check that the homotopy $H$ described descends to $Q$.  If we choose a hyperbolic metric on $\Sigma$ which is preserved by the involution $\tau$, it is also the case that he straight-line homotopies of the $g_t'$ to the identity descend to $Q$.

  The above homotopies show that the quotient map $\overline g\from Q\to Q$ is homotopic rel boundary to the identity on $Q$.  Again applying~\cite[Theorem 7.1]{Waldhausen} we conclude that $\overline g$ is \emph{isotopic} to the identity rel boundary.  Since $\overline g$ lifts to $\Sigma\times I$, so does the isotopy.
\end{proof}

\begin{proof}[Proof of Proposition \ref{prop:separating}]
  We have established in Lemma~\ref{lem:finiteorder} that $M$ is double covered by a finite order mapping torus, so $M$ is Seifert fibered.  In fact, the fibers of this fibration come from the $I$--bundle structures on the $C_{\pm 1}$ described in Lemma~\ref{lem:semibundle}.  In particular, $\Sigma$ is horizontal.  We also see from this description that the base orbifold is $\mathcal{O} = \Sigma / \langle \tau_{\pm 1}\rangle$, where $\tau_{\pm 1}$ are the involutions from Lemma~\ref{lem:semibundle}.  Since $\psi^2$ commutes with both $\tau_1$ and $\tau_2$, Lemma~\ref{lem:descends} implies that $\mathcal{O}$ is large as required.

  It remains to show that $f$ is isotopic rel $\Sigma$ to a partially pseudo-Anosov homeomorphism.  We express $M$ in coordinates, as a quotient of the space $\widehat{M\mminus \Sigma}$ used in the proof of Lemma~\ref{lem:semibundle}.  Namely
  \begin{equation} \label{eq:tilde}
    M = \Sigma \times [0,1] \times \{ \pm 1 \} / \sim,
  \end{equation}
  where $\sim$ is the equivalence relation generated by the deck transformation $(x,t,\epsilon) \sim r(x,t,\epsilon) = (\tau_\epsilon,1-t,\epsilon)$ together with the gluing of the zero levels $(x,0,\epsilon)\sim (x,0,-\epsilon)$.  Because $f(\Sigma) = \Sigma$, there is an induced map
  \[ \tilde f \from \Sigma \times [0,1]\times\{\pm 1\} \to \Sigma \times [0,1]\times\{\pm 1\}, \]
which commutes with $r$.
  
When $f$ preserves the sides of $\Sigma$ in $M$,
$\tilde f$ preserves the components of $\Sigma \times [0,1]\times\{\pm 1\}$.  We have
\begin{equation*}
  \tilde f(x,0,\epsilon) = (\psi(x),0,\epsilon),
\end{equation*}
and
\begin{equation*}
  \tilde f(x,1, \epsilon) = r \tilde f r (x,1,\epsilon) = (\tau_\epsilon \psi \tau_\epsilon(x),1,\epsilon) = (\psi(x),1,\epsilon).
\end{equation*}
The last equality is from Lemma~\ref{lem:semibundle}.\eqref{itm:taucommute}.
Let $h(x,t,\epsilon) = (\psi(x),t,\epsilon)$, and note that $h$ respects the equivalence relation in~\eqref{eq:tilde} so it descends to a partially pseudo-Anosov homeomorphism $\bar{h}\from M\to M$ preserving the sides of $\Sigma$.  We show that $f$ is isotopic rel $\Sigma$ to $\bar{h}$.
Indeed $h^{-1}\circ \tilde f$ restricts to the identity on $\partial(\Sigma \times [0,1]\times\{\pm 1\})$ so
Lemma~\ref{lem:wald1} implies that it is isotopic rel boundary to the identity.  Since $\tilde f$ and $h$ both commute with $r$, Lemma~\ref{lem:wald2} implies that the isotopy descends to $M\mminus \Sigma$.  The surface $\Sigma$ is fixed during this isotopy so we deduce that $\bar h^{-1}\circ f$ is isotopic rel $\Sigma$ to the identity, and finally that $f$ is isotopic to the partially pseudo-Anosov map $\bar h$.  This establishes the Proposition in case $f$ preserves the sides of $\Sigma$.

Suppose now that $f$ exchanges the sides of $\Sigma$, so that $\tilde f$ exchanges the components of $\Sigma \times [0,1]\times\{\pm 1\}$.  For each $\epsilon$ we have
\[ \tilde f(x,0,\epsilon) = (\psi(x),0,-\epsilon)\]
and
\begin{equation*}
  \tilde f(x,1,\epsilon) = r\tilde f r(x,1,\epsilon) = (\tau_{-\epsilon}\psi\tau_\epsilon(x),1,-\epsilon) = (\psi(x),1,-\epsilon).
\end{equation*}
The last equality is from Lemma~\ref{lem:semibundle}.\eqref{itm:tauswitch}.
Let $h(x,t,\epsilon) = (\psi(x),t,-\epsilon)$.  Again $h$ respects the equivalence relation in~\eqref{eq:tilde} and descends to a partially pseudo-Anosov homeomorphism $\bar{h}\from M\to M$, this time exchanging the sides of $\Sigma$.  Again we have $h^{-1}\circ\tilde f$ equal to the identity on $\partial(\Sigma \times [0,1]\times\{\pm 1\})$, so Lemma~\ref{lem:wald1} implies that it is isotopic rel boundary to the identity.  As before this implies that $f$ is isotopic rel $\Sigma$ to $\bar h$, establishing the Proposition in the case that $f$ exchanges the sides of $\Sigma$.
\end{proof}

Putting together Propositions~\ref{prop:nonseparating} and~\ref{prop:separating}, we have finished the proof of Theorem~\ref{thm:necessary}.

\section{Examples of incompressible pseudo-Anosov surfaces}\label{sec:existence}

We now give examples of maps realizing the surfaces described in Theorem~\ref{thm:classify} as pseudo-Anosov surfaces, completing the proof of that theorem.

\subsection{Mapping tori of pseudo-Anosov homeomorphisms}
We begin with the following elementary observation: 

\begin{lemma}\label{lem:commute implies pA}
  Let $\varphi$ and $\psi$ be homeomorphisms of the same surface which commute in the mapping class group (i.e. $\varphi\circ\psi\simeq \psi\circ\varphi$).  Let $M = M_\varphi$ be the mapping torus determined by $\varphi$, and let $\Sigma$ be a fiber of this mapping torus.  Then there is a homeomorphism of $M$ which restricts to $\psi$ on $\Sigma$.
\end{lemma}
\begin{proof}
  Given a homeomorphism $\varphi\from F \to F$ we describe the mapping torus as
  \[ M_\varphi = F \times [0,1] / (x,0)\sim (\varphi(x),1) .\]
  We take $\Sigma$ to be the fiber which is the image of $F\times 0$.  Let $H_t$ be an isotopy from $\psi$ to $\varphi \psi \varphi^{-1}$.  Define $f(x,t) = (H_t(x),t)$.  We verify this is well-defined:  for $x\in F$, we have
  \[f(x,0) = (\psi(x), 0 ) \sim (\varphi\psi(x),1) = (\varphi\psi\varphi^{-1}\varphi(x),1) = f(\varphi(x),1).\]
\end{proof}

Therefore, if $\Sigma$ is the fiber of a mapping torus $M_\varphi$, where $\varphi$ is pseudo-Anosov, then Lemma \ref{lem:commute implies pA} implies that  $\Sigma$ is a pseudo-Anosov surface, by taking $\psi = \varphi$. This covers item \eqref{itm:pamt} of Theorem \ref{thm:classify}.

\subsection{Seifert fibered examples}
We now give examples of homeomorphisms realizing surfaces as in~\eqref{itm:sfs} of Theorem~\ref{thm:classify}.

For the remainder of the subsection we fix the following:

\begin{assumption}
  $M$ is a Seifert fibered manifold with a two-sided horizontal surface $\Sigma$.
\end{assumption}

We first observe that $M$ must be the quotient of a product by a cyclic or finite dihedral group.

\begin{lemma}\label{lem:deck}
    There is a regular finite sheeted cover $\hat{M}\cong \Sigma\times S^1$ with deck group $\tilde G$ of one of the following two types (writing $S^1$ as $\bR/2\pi\bZ$):
  \begin{enumerate}
  \item For some $k>0$, there is an order $k$ automorphism $\varphi$ of $\Sigma$ so that $\tilde G$ is generated by $(x,\theta)\mapsto (\varphi(x),\theta + \frac{2\pi}{k})$.
  \item There is a pair of involutions $\tau_1,\tau_2$ of $\Sigma$ so that $\tau_1\tau_2$ has order $k$, and $\tilde G$ is generated by
    \[ (x,\theta)\mapsto (\tau_1(x), - \theta)\] and
    \[ (x,\theta)\mapsto (\tau_2(x),\frac{2\pi}{k} - \theta).\]
 \end{enumerate}
 The horizontal surface $\Sigma\subset M$ is the image of $\Sigma\times 0$ in $\Sigma\times S^1$.  Let $G$ be the subgroup of $\Homeo(\Sigma)$ obtained by ignoring the action of $\tilde G$ on the circle factor.
 The base orbifold of $M$ can be identified with $\Sigma/G$.
\end{lemma}
\begin{proof}
  There are two cases, depending on whether the horizontal surface $\Sigma$ is separating.

  In case $\Sigma$ is non-separating, the manifold $M$ is a mapping torus $M_\varphi$ where $\varphi$ is given by the first return map of the flow along the fibers.  The map $\varphi$ is well-defined and unique up to taking inverses, since $\Sigma$ is assumed to be two-sided.  The surface $\Sigma$ cuts a regular fiber into $k$ components for some $k\ge 1$, and this is the order of $\varphi$.  In particular, the $\bZ$--action on $\Sigma\times S^1$ via $(x,\theta) \mapsto (\varphi(x), \theta +\frac{2\pi}{k})$ descends to an action of $\bZ/k$ on $\Sigma\times S^1$.  The reader can easily see that the action is free with quotient $M$.

  Otherwise, $\Sigma$ separates $M$ into a pair of nontrivial $I$--bundles $M_1$ and $M_2$.  For each $i$ we obtain a first-return map $\tau_i\from \Sigma\to \Sigma$ obtained by following the fiber into the corresponding $I$--bundle.  Note that $\tau_i$ is a fixed-point free involution of $\Sigma$.  The surface $\Sigma$ cuts a regular fiber into $2k$ components, for some $k\ge 1$, and we see that $\tau_2\tau_1$ has order $k$.  Intersection with the zero-sections of the $I$--bundles gives a homomorphism to the dihedral group of order $2k$.  The cover corresponding to the kernel is a product, and the deck group acts freely as in the statement of the lemma.
\end{proof}

\begin{lemma}\label{lem:descend}
  If $\psi$ is a homeomorphism of $\Sigma$ which commutes with every element of $G$, where $G$ is as in Lemma \ref{lem:deck}, then the self homeomorphism of $\Sigma\times S^1$ defined by $\Psi(x,\theta)=(\psi(x),\theta)$ descends to a homeomorphism of $M$ which restricts to $\psi$ on $\Sigma$.
\end{lemma}
\begin{proof}
   We just need to check that the elements of $\tilde{G}$ commute with $\Psi$.  This is immediate, since on the $\Sigma$ factor the element $\psi$ commutes with $G$, and on the circle factor the identity commutes with everything.
\end{proof}

\begin{proposition}\label{prop:existSF}
  If $\Sigma/G$ is large, where $G$ is as in Lemma \ref{lem:deck}, then $\Sigma$ is a pseudo-Anosov surface in $M$.  Moreover, $\Sigma$ is realized by a partially pseudo-Anosov homeomorphism of $M$.
\end{proposition}
\begin{proof}
  Let $\bar\psi\from \Sigma/G\to \Sigma/G$ be a pseudo-Anosov map.  Replacing $\bar\psi$ by a power if necessary, we may assume that $\bar\psi$ lifts to a pseudo-Anosov $\psi\from \Sigma\to \Sigma$.  

  Since $\psi$ is a lift of a map on $\Sigma/G$, we have that for any $g\in G$ there is some $g'\in G$ so that $g \psi = \psi g'$, or $\psi^{-1} g \psi = g'$.  In particular conjugation by $\psi$ acts as some automorphism of the finite group $G$.  Raising $\psi$ (and $\bar\psi$) to a positive power we may assume that $\psi$ commutes with every element of $G$.

  We define an automorphism $\Psi$ of $\Sigma \times S^1$ by $\Psi(x,t) = (\psi(x),t)$.  By Lemma~\ref{lem:descend}, this descends to an automorphism of $M$ with the desired properties.
\end{proof}

\section{Reducible 3-manifolds (Proof of Corollary~\ref{c:main})}\label{sec:reducible}

Suppose now $M$ is a closed reducible 3-manifold. By the Kneser-Milnor decomposition theorem, $M$ can be decomposed as a finite connected sum 
 $$M = M_ 1 \# M_ 2  \# \cdots \# M_n\#(\#_m S^1\times S^2),$$
 where each $M_i$ is irreducible and $m\geq 0$. 
 
 The following fundamental theorem on the mapping class groups of reducible 3-manifolds was proved by McCullough.

\begin{theorem} {\normalfont(\cite[pg. 69]{Mc}).} \label{Mc1} 
Let $M$ be a closed oriented connected 3-manifold. Any orientation-preserving homeomorphism of $M$ is  isotopic to a composition of the following types of homeomorphisms: 
\begin{itemize}
\item[(1)] homeomorphisms  preserving  summands;
\item[(2)] interchanges of homeomorphic summands;
\item[(3)] spins of $S^ 1 \times S^ 2$ summands;
\item[(4)] slide homeomorphisms.
\end{itemize}
\end{theorem}

We refer to~\cite{Mc} for the description of the four types of homeomorphisms that appear in  above theorem. The proof of Theorem \ref{Mc1} contains in fact the following statement: 

\begin{theorem} \label{Mc2} 
If $M$ is a closed oriented connected 3-manifold and $f$ is an orientation-preserving homeomorphism of $M$, then 
$$hf  =g_3 g_2 g_1,$$ 
where $h$ is a finite composition of homeomorphisms of type $(4)$ (i.e., slide homeomorphisms) and isotopies on $M$, and each $g_k$ is a composition of finitely many homeomorphisms of $M$ of type $(k)$, $k=1,2,3$. 
\end{theorem}

Theorem \ref{Mc2} alone implies that $hf$ permutes the prime summands of $M$. The main result of~\cite{NeoWang} is that $h$ can be chosen so that its restriction on each $f$-invariant incompressible surface is the identity~\cite[Theorem 1.3]{NeoWang}. An immediate consequence of that result and of its proof is the following:

\begin{corollary}\label{preserve summands}
Let  $f$ be an orientation-preserving homeomorphism of a closed oriented connected 3-manifold $M$ and $F\subset M$ an incompressible surface with $f(F)=F$.
Then $F$ can be isotoped into a prime summand of $M$ so that the homeomorphism $hf$ of $M$ 
preserves this prime summand and $F$, 
where $h$ is a finite composition of slide homeomorphisms and isotopies.
\end{corollary}

Therefore, once we know a complete list of prime 3-manifolds that admit a pseudo-Anosov incompressible surface, then we derive a complete list of 3-manifolds admitting such a surface:

\begin{proposition}\label{p:reducible}
A closed oriented reducible 3-manifold admits a pseudo-Anosov surface if and only if one of its prime summands admits such a surface.
\end{proposition}
\begin{proof}
  The ``only if'' part was discussed above. The converse can easily be shown: Let $\Sigma$ be a pseudo-Anosov surface in a closed oriented 3-manifold $M$, and $f\colon M\to M$ be a homeomorphism so that $f\vert_\Sigma$ is pseudo-Anosov.  After possibly replacing $f$ with its square, we may assume that $f$ preserves the components of $M\mminus \Sigma$ and preserves the orientation on $M$.  This implies that $f$ is isotopic rel $\Sigma$ to a map $\widetilde{f}$ which is the identity on a smooth closed ball $B$ disjoint from $\Sigma$.
  Let $N$ be any closed oriented 3-manifold, let $B'$ be any smooth closed ball in $N$.  Use the balls $B,B'$ to form the connect sum:
  \[ M\#N=(M\setminus \Interior{B})\cup_{\partial B\cong\partial B'}(N\setminus \Interior{B'}). \]
  Define  
\[
g\colon M\#N\longrightarrow M\#N
\]
by $g\vert_{M\setminus \Interior{B}}:=\widetilde f$ and $g\vert_{N\setminus \Interior{B'}}=\mathrm{id}$.  Then $g$ is a homeomorphism and $g\vert_\Sigma=f\vert_\Sigma$ is pseudo-Anosov.
\end{proof}

The proof of Corollary \ref{c:main} is now complete.

 \section{Compressible pseudo-Anosov surfaces}\label{sec:compressible}
By a \emph{compressible pseudo-Anosov} surface in a three-manifold $M$, we mean a two-sided, non-$\pi_1$--injective surface $\Sigma\subseteq M$ so that there is a homeomorphism of pairs $h\from (M,\Sigma)\to (M,\Sigma)$ restricting to a pseudo-Anosov map on $\Sigma$.
In contrast to the incompressible case, compressible pseudo-Anosov surfaces are extremely common.

\begin{lemma}
  The standard genus 2 Heegaard surface in $S^3$ is compressible pseudo-Anosov. 
\end{lemma}
\begin{proof}
  Let $\alpha$ be the simple closed curve pictured in Figure~\ref{fig:genus2curve}.
  \begin{figure}[htbp]
    \hspace*{3cm}\includegraphics{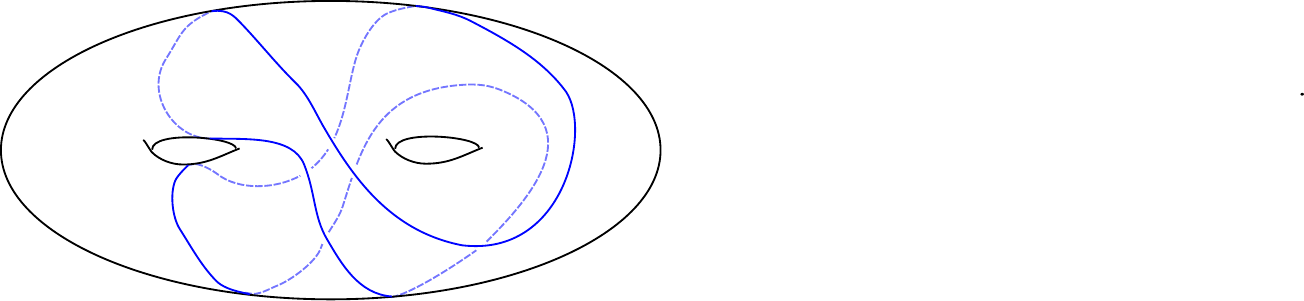}
    \caption{The curve $\alpha$.}
    \label{fig:genus2curve}
  \end{figure}
  Let $\beta$ be the curve obtained from $\alpha$ by reflecting the picture left-to-right.  The reader may verify the following:
  \begin{itemize}
  \item $\alpha$ is the transverse intersection of a $2$--sphere with $\Sigma$ (and hence so is $\beta)$.
  \item The pair of curves $\alpha,\beta$ fill the surface $\Sigma$.
  \end{itemize}
  For any sphere $S$ transversely intersecting the Heegaard surface $\Sigma$ in a single curve, there is a homeomorphism of $S^3$ supported in that sphere, preserving $\Sigma$ and restricting to a Dehn twist around $\Sigma\cap S$.  Indeed, there is a neighborhood $N$ of $S$ homeomorphic to $S\times I$, by a homeomorphism taking $\Sigma\cap N$ to $\alpha\times I$ for $\alpha$ a simple closed curve on $S$.  Think of $S$ as the round $2$--sphere in $\bR^3$, and $\alpha$ as the equator.  Let $r_t$ be rotation in angle $2\pi t$ preserving the equator, and define $T \from N \to N$ by $T(x,t) = (r_t(x),t)$.  The map $T$ restricts to the identity on $\partial N$, and to a Dehn twist in the annulus $N\cap \Sigma$.  (See \cite{Sch04} for a fuller discussion of the group of isotopy classes of homeomorphisms of $S^3$ preserving a genus two Heegaard splitting, also called the \emph{Goeritz group} of that splitting.)

  By a result of Penner~\cite{Penner}, the product of a positive Dehn twist around $\alpha$ with a negative Dehn twist around $\beta$ is isotopic to a pseudo-Anosov homeomorphism $\psi$ of $\Sigma$.  This isotopy can be extended to a small neighborhood of $\Sigma$, and so we obtain a homeomorphism $f\from S^3\to S^3$, preserving $\Sigma$, so that $f|\Sigma$ is the pseudo-Anosov map $\psi$.
\end{proof}

Since the homeomorphism $f$ constructed in the proof can be chosen to be the identity on a ball disjoint from $\Sigma$, we can embed this example inside any $3$--manifold at all, and we obtain:

\begin{corollary}
  Every $3$--manifold contains a compressible genus 2 pseudo-Anosov surface.
\end{corollary}

A complete classification of compressible pseudo-Anosov surfaces seems difficult.
As an example, for Heegaard surfaces one is required to answer the following question.

\begin{question}
  For what Heegaard splittings $(M,\Sigma)$ does the Goeritz group contain a pseudo-Anosov homeomorphism?
\end{question}

\end{document}